\documentclass[12pt]{article}

\usepackage[utf8]{inputenc} 
\usepackage[T1]{fontenc}
\usepackage{url}
\usepackage{ifthen}
\usepackage{cite}
\usepackage{wrapfig}

\usepackage[cmex10]{amsmath}
\usepackage{macros}

\begin{document}
\title{Robust Alignment via Partial\\ Gromov-Wasserstein Distances}

\author{Xiaoyun Gong \and Sloan Nietert \and Ziv Goldfeld\thanks{Cornell University, \texttt{xg332@cornell.edu}, \texttt{nietert@cs.cornell.edu}, \texttt{goldfeld@cornell.edu}}}

\maketitle

\begin{abstract}
    The Gromov-Wasserstein (GW) problem provides a powerful framework for aligning heterogeneous datasets by matching their internal structures in a way that minimizes distortion. However, GW alignment is sensitive to data contamination by outliers, which can greatly distort the resulting matching scheme. To address this issue, we study robust GW alignment, where upon observing contaminated versions of the clean data distributions, our goal is to accurately estimate the GW alignment cost between the original (uncontaminated) measures. We propose an estimator based on the partial GW distance, which trims out a fraction of the mass from each distribution before optimally aligning the rest. The estimator is shown to be minimax optimal in the population setting and is near-optimal in the finite-sample regime, where the optimality gap originates only from the suboptimality of the plug-in estimator in the empirical estimation setting (i.e., without contamination). Towards the analysis, we derive new structural results pertaining to the approximate pseudo-metric structure of the partial GW distance. Overall, our results endow the partial GW distance with an operational meaning by posing it as a robust surrogate of the classical distance when the observed data may be contaminated.
\end{abstract}

\section{Introduction}
\label{sec:intro}

Alignment of heterogeneous datasets with varying modalities or semantics is fundamental to data science. The Gromov-Wasserstein (GW) distance~\citep{memoli2011gromov} offers~a mathematical framework for alignment by modeling datasets as metric measure (mm) spaces $(\cX,\mathsf{d}_\cX,\mu)$ and $(\cY,\mathsf{d}_\cY,\nu)$ and 
optimally~matching~them:
\begin{equation}
    \mathsf{GW}_{p,q}(\mu,\nu)\coloneqq \mspace{-3mu}\inf_{\pi\in\Pi(\mu,\nu)}\mspace{-6mu}\mathbb{E}_{\pi\otimes\pi}\big[%
 \Delta_q\big((X,X'),(Y,Y')\big)^p\big]^{\frac 1p},\label{eq:GW_def}
\end{equation}
where $\Pi(\mu,\nu)$ is the set of all couplings of $\mu,\nu$ and $\Delta_q\big((x,x'),(y,y')\big)\coloneqq\big| \mathsf{d}_\cX(x,x')^q-\mathsf{d}_\cY(y,y')^q\big|$ is the distance distortion cost to be minimized; see \cref{fig:mm_spaces}. An optimal coupling $\pi^\star$ serves as an alignment plan that best preserves the pairwise distance distribution. The GW distance serves as an optimal transport (OT)-based $L^p$ relaxation of the classical Gromov-Hausdorff distance and defines a metric on the space of all mm spaces modulo the equivalence induced by isomorphism.\smallskip

\begin{figure}[!t]
\begin{center}
\includegraphics[width=0.43\textwidth]{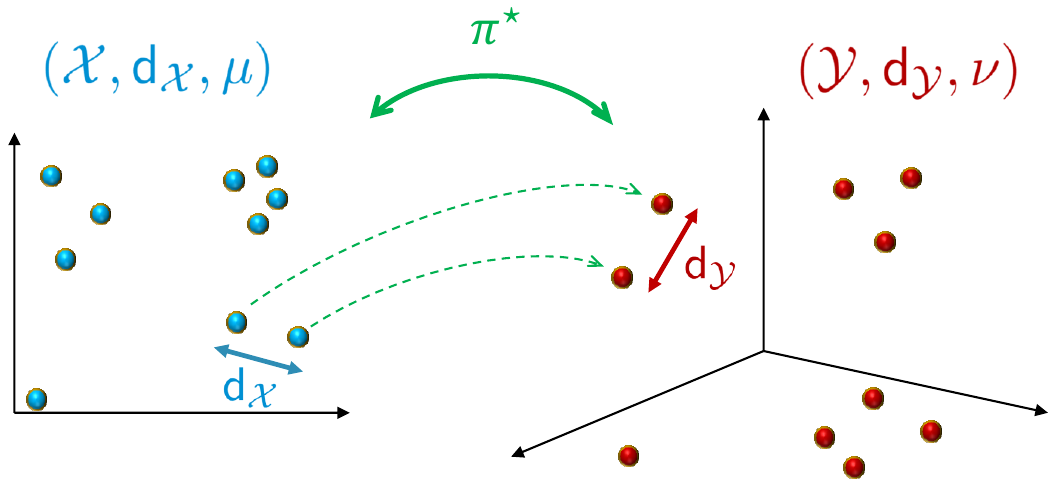}    
\end{center}
\caption{Mm space representation of datasets, where the metric structure encodes the geometry of each space, while the probability measure captures the density/intensity/weight of different points. An optimal GW alignment plan $\pi^\star\in\Pi(\mu,\nu)$ provides a matching that minimizes distance distortion.}\label{fig:mm_spaces}
\end{figure}

The GW framework has seen widespread usage in tasks involving alignment of heterogeneous datasets, spanning single-cell genomics \citep{blumberg2020mrec, cao2022manifold, demetci2022scot}, alignment of language models \citep{alvarez2018gromov}, shape matching \citep{koehl2023computing,memoli2009spectral}, graph matching \citep{chen2020graph,petric2019got,xu2019gromov,xu2019scalable}, heterogeneous domain adaptation \citep{sejourne2021unbalanced,yan2018semi}, and generative modeling \citep{bunne2019learning}. However, in practical scenarios, data can deviate from distributional assumptions due to measurement error, model misspecication, or even data poisoning attacks. The GW distance is highly sensitive to such corruptions as even a small outlier mass at a far away location can greatly distort the pairwise distance distribution, potentially affecting both the alignment plan and the resulting cost. In particular, $\GW(\mu,(1-\eps)\mu + \eps\delta_a) \to \infty$ as $\|a\| \to \infty$, while $\GW(\mu,\mu)$ clearly nullifies for any $\mu$. %
In this work, we focus on making the distance computation more robust to outliers, without explicitly addressing the robustness of the alignment plan itself. To that end, we study robust alignment via the partial GW problem, providing minimax optimal rates and risk bounds.

\subsection{Contributions}

We model data corruption as an $\eps$-perturbation of the original (clean) $\mu,\nu$ distributions in the total variation (TV) distance. Such TV perturbations allow for an $\eps$ amount of mass to be arbitrary relocated, capturing the injection of possibly adversarial outliers. Upon observing the corrupted distributions, our goal is to estimate the GW distance between the original $\mu,\nu$ in a minimax-optimal manner. For concreteness, we henceforth assume $p=q=2$, also known as the \emph{quadratic GW problem}, and omit the subscripts from our notation, simply writing $\GW(\mu,\nu)$. The key technical tool we employ for this robust estimation task is the \emph{partial GW distance}. Denoted $\PGW$, partial GW is defined akin to \eqref{eq:GW_def} up to replacing $\Pi(\mu,\nu)$ with the set of \emph{partial} couplings $\Pi^{\eps}(\mu,\nu) = \{ \pi\in \cM_+(\cX\times \cY): \pi(\cdot , \mathcal{Y}) \leq \mu, \pi(\mathcal{X}, \cdot) \leq \nu, \pi(\mathcal{X}, \mathcal{Y}) = 1-\eps \}$, where $\cM_+(\cX\times \cY)$ denotes the set of nonnegative measures on the product space $\cX\times \cY$. Consequently, $\PGW$ excludes an $\eps$ amount of mass for each space (deemed the outliers) and optimally aligns the~rest.\smallskip

Assuming that the clean $\mu,\nu$ belong to some family encoding distributional assumptions (e.g., bounded moments), we treat both the population limit and the finite-sample regimes. In the population limit setting, when no sampling is involved, given the TV $\eps$-contaminated versions $\tilde\mu,\tilde\nu$, we show that the partial GW distance $\PGW(\tilde\mu,\tilde\nu)$ is a minimax optimal robust estimator of the original $\GW(\mu,\nu)$ and characterize the corresponding estimation risk. To that end, we establish new structural properties of the partial GW distance, in particular, showing that it is an approximate pseudo-metric.  %
We also treat the finite-sample setting, where the clean populations $\mu,\nu$ are first sampled before the corruption mechanism manipulates an $\eps$ fraction of the data points. The finite-sample risk is shown to decompose into the population limit risk plus the error due to sampling (without corruptions), with the gap between our upper and lower bounds only coming from the suboptimality of the plug-in estimator of the GW distance between the populations. Overall, our results establish partial GW as an appropriate surrogate for alignment in the presence of (possibly) adversarial outliers.

\subsection{Related Work}
The notion of partial GW alignment has appeared in the literature under several definitions, some differing from ours. In \citet{kong2024outlier,bai2024efficient}, the strict marginal constraints from \eqref{eq:GW_def} are relaxed by allowing $\pi$ to be an arbitrary non-negative measure on $\cX\times\cY$ and penalizing the objective by Kullback-Leibler (KL) divergence or TV distance terms between the marginals of $\pi$ and the given $\mu,\nu$. \cite{chhoa2024metric} employed a definition similar to ours based on partial couplings, and explored topological and metric properties of the partial GW distance. Computational aspects of partial GW alignment were addressed in \cite{chapel2020partial} by providing an efficient Frank-Wolfe algorithm for evaluating  $\PGW$. A framework for partial graph matching that combines partial GW and partial OT was proposed in \cite{liu2020partial}.\smallskip

Also related is the body of work on outlier-robust OT \citep{balaji2020robust,mukherjee2021outlier,nietert2022outlier,nietert2023robust}.
In particular, \cite{nietert2023robust} considered robust distribution estimation under the Wasserstein distance, which is closely related to the question of estimating the distance itself from corrupted observations. Our modeling assumption of TV $\eps$-corruption is inspired by that work, as well as some of the analysis techniques. Nevertheless, there are some key differences between robust estimation of OT and GW, which we highlight in the sequel. In particular, while partial GW offers a minimax-optimal estimator in the alignment setting, partial OT is strictly suboptimal. This discrepancy originates from the translation invariance of the GW distance which mitigates the effect of outliers in a way that allows establishing optimality.  

\medskip

\noindent\textbf{Notation.} Let $\|\cdot\|$ and $\langle \cdot, \cdot \rangle$ designate Euclidean norm and the inner product in $\mathbb{R}^d$, respectively. The class of Borel probability measures on $\mathcal{X} \subseteq \mathbb{R}^d$ is denoted by $\mathcal{P}(\mathcal{X})$, while $\cM_+(\cX)$ represents the set of nonnegative measures on $\cX$. For $\mu,\nu\in\cM_+(\cX)$, we write $\|\mu-\nu\|_\tv$ for the TV norm of their difference. Inequalities between nonnegative measures $\mu\leq \nu$ are understood in the set-wise sense, i.e., $\mu(A)\leq \nu(A),\,\forall A$. We use $\lesssim_x$ to denote inequalities up to constants that only depend on $x$; the subscript is dropped when the constant is universal. When both $f\lesssim g$ and $f\gtrsim g$, we write $f\asymp g$. For $a,b \in \mathbb{R}$, we write $a \vee b = \max \{ a,b \}$ and $a \land b = \min \{ a,b \}$.

\section{Robust Estimation of GW Alignment}
\label{sec:results}

In this section, we set up the robust estimation problem, present our main results, and provide corresponding discussion. All proofs are deferred to \cref{sec:proofs}.

\subsection{Metric structure}

An important technical result we employ is the fact that partial GW forms an approximate pseudo-metric. One readily observes that $\PGW$ is not a metric: it nullifies when $\|\mu - \nu\|_\tv \leq \eps$, even if $\mu\neq\nu$, and violates the triangle inequality. For the latter, consider $\mu = \delta_0$, $\nu = (1-\eps)\delta_0+\eps \delta_1$, and $\kappa = (1-2\eps)\delta_0+\eps \delta_1+\eps \delta_2$, for which $\PGW(\mu,\nu) = \PGW(\nu,\kappa)=0$ but $\PGW(\mu,\kappa)>0$.

\begin{proposition}[Approximate pseudo-metric]
\label{prop:metric}
For $\eps,\delta \in [0,1]$ and measures $\mu \!\in\! \cP(\cX),\, \nu \!\in\! \cP(\cY),\, \kappa \!\in\! \cP(\cZ)$, we~have
\begin{enumerate}[(i)]
    \item $\PGW(\mu,\nu) \geq 0$ with equality if and only if there exists an isometry $T:\cX \to \cY$ such that $\|T_\sharp \mu - \nu\|_\tv \leq \eps$;
    \item $\PGW(\mu,\nu) = \PGW(\nu,\mu)$;
    \item $\GW^{\eps + \delta}(\mu,\nu) \leq \PGW(\mu,\kappa) + \GW^{\delta}(\kappa,\nu)$.
\end{enumerate}
\end{proposition}

The first two properties follow rather directly from the definition of $\PGW$. The last result is obtained by gluing together two feasible partial couplings for the $\PGW$ and $\GW^{\delta}$ problems into a feasible partial coupling for the $\GW^{\eps + \delta}$ problem.

\subsection{Robust Estimation Setting and Main Results}%

We now examine robust alignment, first in the population (infinite-sample) limit and then in the finite-sample regime. 

\medskip
\noindent\textbf{Population limit.} 
Beginning with clean measures $\mu$ and $\nu$, we face a TV-constrained adversary that produces contaminated measures $\tilde{\mu},\tilde{\nu}$ satisfying $\|\tilde{\mu} - \mu\|_\tv, \|\tilde{\nu} - \nu\|_{\tv} \leq \eps$. We aim to recover $\GW(\mu,\nu)$ from $\tilde{\mu},\tilde{\nu}$ using an estimator $\sfT$. Given families $\cG_1, \cG_2 \subseteq \cP(\R^d)$ encoding prior knowledge of $\mu$ and $\nu$, respectively, the worst-case risk incurred by $\sfT$ is
\begin{equation*}
    R_{\infty}(\sfT,\cG_1,\cG_2,\eps) \defeq \mspace{-5mu}\sup_{\substack{\mu\in\cG_1\\\nu \in\cG_2}}\sup_{\substack{\tilde{\mu}, \tilde{\nu}:\\
    \|\tilde{\mu} - \mu\|_\tv\leq \eps\\ \|\tilde{\nu} - \nu\|_{\tv} \leq \eps}}\mspace{-10mu} |\sfT(\tilde{\mu},\tilde{\nu}) - \GW(\mu,\nu)|,
\end{equation*}
with the minimax robust estimation risk then given by 
\[
R_{\infty}(\cG_1,\cG_2,\eps) \defeq \inf_{\sfT}R_{\infty}(\sfT,\cG_1,\cG_2,\eps).
\]
This setting is termed the \emph{population limit} since no sampling occurs.
When $\cG_1 = \cG_2 = \cG$, dubbed the \emph{symmetric case}, we only write the distribution family once in the risk notation. \smallskip

Our main result characterizes the minimax risk for the bounded moment family $\cG_k(\sigma) \defeq \{\mu \in \cP(\R^d) : \mathbb{E}_{\mu}[\|X\|^k] \leq \sigma^k\}$ and establishes optimality of the partial GW estimator. 

\begin{theorem}
\label{thm:sym_pop-limit}
For $\eps \in [0, 0.33]$\footnote{The risk bound for $\PGW$ holds for any $\eps$ bounded away from $1/3$; we fix the upper bound to suppress nuisance dependencies on $(1-3\eps)^{-1}$. The minimax risk bound still holds for any $\eps$ bounded away from $1/2$, via a TV projection procedure of \cite{donoho1988automatic}. See \Cref{rem:breakdown-point,rem:tv-projection} for further discussion.}, $\sigma \geq 0$, and $k \geq 4$, we have
\begin{align*}
    R_{\infty}(\PGW, \cG_{k}(\sigma),\eps) &\asymp R_{\infty}(\cG_k(\sigma),\eps) \asymp \sigma^2 \eps^{\frac{1}{2} - \frac{2}{k}}.
\end{align*}
\end{theorem}

For the upper bound, we first bound the estimation error through resilience, a standard notion from robust statistics, and then bound resilience using Hölder's inequality. We employ a two-point lower bound, constructing two pairs of clean measures such that no estimator can perform well on both simultaneously, but from which the same corrupted observation can be generated (making them indistinguishable).

\begin{remark}[Comparison to OT]
 \Cref{thm:sym_pop-limit} shows that $\PGW$ achieves the optimal rate for $\cG_k(\sigma)$, the family of measures that satisfy a moment bound centered around zero. In fact, the centering around zero (or any point) is not necessary due to the translation-invariance of $\PGW$. In the Wasserstein case, however, partial Wasserstein estimation is only optimal if $\mu$ and $\nu$ have bounded moments around the same point \citep{nietert2023robust}.
\end{remark}

We now consider the families of sub-Gaussian distributions and those with bounded sliced $k$th moments:
\begin{align*}
    \cG_{\mathrm{subG}}(\sigma) &\mspace{-3mu}\defeq\mspace{-3mu} \Big\{\mu\mspace{-1mu}\in\mspace{-1mu}\mathcal{P}(\R^d)\mspace{-3mu}:\mathbb{E}_\mu\Big[e^{\frac{\langle \theta, X - \mathbb{E}_\mu[X] \rangle^2}{\sigma^2}}\Big] \mspace{-2mu}\leq\mspace{-2mu} 2,\, \forall \theta\mspace{-1mu}\in\mspace{-1mu} \unitsph\!\Big\}\\
    \overline{\cG}_k(\sigma) &\mspace{-3mu}\defeq\mspace{-3mu} \big\{\mu\mspace{-1mu}\in\mspace{-1mu}\mathcal{P}(\R^d)\mspace{-3mu}:\mathbb{E}_\mu[|\langle X,\theta \rangle|^k] \leq \sigma^k,\, \forall \theta\mspace{-1mu}\in\mspace{-1mu} \unitsph\big\}.
\end{align*}
\cref{thm:sym_pop-limit} implies tight risk bounds for these classes via inclusions in $\cG_k(\sigma)$. 

\begin{corollary}
\label{cor:pop-limit-examples}
For $\eps \in [0, 0.33]$, $\sigma \geq 0$, and $k \geq 4$,
\begin{align*}
    &R_{\infty}\mspace{-1mu}(\mspace{-1mu}\PGW\mspace{-4mu},\mspace{-1mu} \cG_{\mathrm{subG}}(\mspace{-1mu}\sigma\mspace{-1mu})\mspace{-1mu},\mspace{-1mu}\eps) \mspace{-2mu}\asymp\mspace{-2mu} R_{\infty}\mspace{-1mu}(\mspace{-1mu}\cG_{\mathrm{subG}}(\mspace{-1mu}\sigma\mspace{-1mu}),\eps) \mspace{-2mu}\asymp \mspace{-2 mu}\sigma^2\bigl(d\mspace{-1mu}+\mspace{-1mu}\log\tfrac{1}{\eps}\bigr) \eps^{\frac{1}{2}}\\
    &R_{\infty}(\PGW, \overline{\cG}_k(\sigma),\eps) \asymp  R_{\infty}(\overline{\cG}_k(\sigma),\eps) \asymp \sigma^2 (d/k + 1)\eps^{\frac{1}{2}-\frac{2}{k}}.
\end{align*}
\end{corollary}

\begin{remark}[Breakdown point]
\label{rem:breakdown-point}
We note that $\PGW$ admits no meaningful guarantees for $\eps \geq 1/3$. In this case, for any $\mu$ and $\nu$, one can identify $\eps$-corruptions $\tilde{\mu}$ and $\tilde{\nu}$ such that $\|\tilde{\mu} - \tilde{\nu}\|_\tv \leq 1-2\eps \leq \eps$, implying that $\PGW(\tilde{\mu},\tilde{\nu}) = 0$.
\end{remark}

\begin{remark}[TV projection estimator]
\label{rem:tv-projection}
An alternative estimator, employing the minimum distance method of \cite{donoho1988automatic}, is $\sfT(\tilde{\mu},\tilde{\nu})=\GW(\mu',\nu')$ where $\mu' \in \argmin_{\kappa \in \cG}\|\kappa - \tilde{\mu}\|_\tv$ and  $\nu' \in \argmin_{\kappa \in \cG}\|\kappa - \tilde{\nu}\|_\tv$. Note that this selection ensures that $\|\mu' - \mu\|_\tv \leq \|\mu' - \tilde{\mu}\|_\tv + \|\tilde{\mu} - \mu\|_\tv \leq 2\eps$ and, likewise, $\|\nu' - \nu\|_\tv \leq 2\eps$. From here, a similar resilience argument to that employed in the proof of \cref{thm:sym_pop-limit} shows that $\mathsf{T}$ achieves the minimax risk for the families considered above when $\eps$ is bounded away from $1/2$, beating the breakdown point of $1/3$ for $\PGW$ at the cost of some optimization complexity.
\end{remark}

\medskip

\noindent \textbf{Finite-sample setting.} We now introduce an additional sampling step. Starting with $\mu$ and $\nu$ as above, we first obtain empirical measures $\hat{\mu}_n=\frac
1n\sum_{i=1}^n\delta_{X_i},\hat{\nu}_n=\frac
1n\sum_{i=1}^n\delta_{Y_i}$, where $X_1,\ldots,X_n$ and $Y_1,\ldots,Y_n$ are i.i.d.\ samples from $\mu$ and $\nu$, respectively. The adversary then arbitrarily modifies an $\eps$-fraction of the clean data to produce $\tilde{X}_1, \dots, \tilde{X}_n, \tilde{Y}_1, \dots, \tilde{Y}_n$ such that $\frac{1}{n}\sum_{i=1}\mathds{1}\{\tilde{X}_i \neq X_i\}, \frac{1}{n}\sum_{i=1}\mathds{1}\{\tilde{Y}_i \neq Y_i\} \leq \eps$. Equivalently, the corrupted empirical measures $\tilde{\mu}_n$ and $\tilde{\nu}_n$ satisfy $\|\tilde{\mu}_n - \hat{\mu}_n\|_\tv, \|\tilde{\nu}_n - \hat{\nu}_n\|_{\tv} \leq \eps$. Let $\cM_n^\eps(\mu,\nu)$ denote the family of all joint distributions $P_n$ of the corrupted samples $\{\tilde{X}_i,\tilde{Y}_i\}_{i=1}^n$ which can arise through this process. We then let
\begin{align*}
    R_n(\sfT,\cG_1,\cG_2,\eps) &\mspace{-3mu}\defeq \mspace{-3mu}\sup_{\substack{\mu\in\cG_1\\\nu\in\cG_2}}\sup_{\substack{P_n\in\\  \cM_n^\eps(\mu,\nu)}}\mspace{-5mu}\mathbb{E}_{P_n}[|\sfT(\tilde{\mu}_n,\tilde{\nu}_n)\mspace{-3mu}-\mspace{-3mu}\GW(\mu,\nu)|]
\end{align*}
and define the minimax risk $R_n(\cG_1,\cG_2,\eps)$ analogously to the population limit variant, with the same convention when $\cG_1=\cG_2$. Lastly, we use the notation 
$\tau_n(\cG)\asymp \sup_{\mu \in \cG} \E[\GW(\hat{\mu}_n,\mu)]$ and $\tau_n(\cG_1,\cG_2)\asymp \sup_{\mu \in \cG_1, \nu \in \cG_2} \E[|\GW(\hat{\mu}_n,\hat{\nu}_n)-\GW(\mu,\nu)|]$ for the one-sample (null) and two-sample empirical convergence terms without corruption.\smallskip

We now decouple the errors due to TV corruption and sampling, controlling the finite-sample risk by the population limit risk plus the estimation~error when~$\eps = 0$.

\begin{theorem}
\label{thm:sym_n_sample}
For $\eps \in [0, 0.49]$ and any class $\cG$, it holds that
\begin{equation*}
    R_n(\cG,\eps) \gtrsim R_\infty(\cG,\eps/4) + R_n(\cG,0).
\end{equation*}
If further $\eps \leq 0.33$ and $\cG$ is restricted to any of the families from \cref{thm:sym_pop-limit} or \cref{cor:pop-limit-examples},
we have 
\begin{align*}
    R_n(\PGW, \cG,\eps) &\lesssim R_\infty(\cG,\eps) + \tau_n(\cG).
\end{align*}
\end{theorem}

Note that $R_n(\cG,0)\!\leq\!2\tau_n(\cG)$, since $|\GW(\hat{\mu}_n,\!\hat{\nu}_n)-\GW(\mu,\!\nu)|$ $\leq \GW(\hat{\mu}_n,\mu) + \GW(\hat{\nu}_n,\nu)$. Thus, these bounds match up to the suboptimality of the plug-in estimator in approximating $\GW(\mu,\nu)$ and the application of the triangle inequality.\smallskip %

While fully characterizing $R_n(\cG,0)$ is out of scope, we note that $\tau_n(\cG) \lesssim n^{-1/d}$ under mild assumptions, as stated next.

\begin{proposition}
\label{prop:empirical-convergence}
If $d >8$, $k > 4$, then $\tau_n(\cG_k(\sigma)) \lesssim \sigma^2 n^{-1/d}$.
\end{proposition}

\subsection{The Asymmetric Case}
For the asymmetric case we no longer require $\mu$ and $\nu$ to lie in the same family or have the same ambient dimension. We are given $\mu \in \cG_1 \subseteq \cP(\R^{d_1})$ and $\nu \in \cG_2 \subseteq \cP(\R^{d_2})$.

\begin{theorem}
\label{thm:asym_pop-limit+n_sample}
For $\eps \in [0, 0.49]$, we have
\begin{align*}
    R_n(\cG_1,\cG_2,\eps) &\gtrsim R_\infty(\cG_1, \cG_2,\eps/4) + R_n(\cG_1, \cG_2,0)   
\end{align*}
If further $\eps \leq 0.33$ and $\cG_1, \cG_2$ are restricted to any of the families from \cref{thm:sym_pop-limit} or \cref{cor:pop-limit-examples}, then
\begin{align*}
    R_\infty(\PGW, \cG_1,\cG_2,\eps) &\asymp R_\infty(\cG_1,\cG_2,\eps) \asymp R_\infty(\cG_1 \cup \cG_2,\eps)\\
    R_n(\PGW, \cG_1,\cG_2,\eps) &\lesssim R_\infty(\cG_1 \cup \cG_2,\eps) + \tau_n(\cG_1 \cup \cG_2).
\end{align*}
\end{theorem}

Observe that there is a gap between the upper and lower bounds, at least for $\eps=0$. Specifically, using the plug-in estimator, we have $R_n(\cG_1,\cG_2,0) \leq \tau_n(\cG_1,\cG_2)$, which depends on the smaller of the two dimensions (at least when the measures have bounded support) as shown in \cite{zhang2024gromov}. In contrast, the upper bound in \Cref{thm:asym_pop-limit+n_sample} gives $R_n(\PGW, \cG_1,\cG_2,0) \lesssim \tau_n(\cG_1\cup\cG_2)$, which depends on the larger of the two. It is possible that another approach might refine the upper bound to match the lower bound. Regardless, the population-limit error term is tight and is governed by the larger of the two families. This discrepancy highlights the effect of adversarial contamination, which annuls the adaptive nature of GW estimation rates (with clean data) to the lower-dimensional dataset.

\section{Proofs}
\label{sec:proofs}
For partial couplings $\pi_1, \pi_2 \in \mathcal{M}_+(\mathcal{X}\times \mathcal{Y})$, write $G(\pi_1,\pi_2) \defeq \mathbb{E}_{\pi_1\otimes\pi_2}\big[%
 \Delta_2\big((X,X'),(Y,Y')\big)^2\big]^{1/2}$. 
 We also use $\mu \land \nu$ for the minimum of two measures, i.e.,  $(\mu \land \nu)(A) = \inf_B\{\mu(B)+ \nu(A\setminus B): B\subseteq A, B \text{ measurable}\}$. 
 For $\alpha \in \cM_+(\cX)$ and integrable $f:\cX \to \R$, we write $\alpha(f) \defeq \int f \dd \alpha$ and reserve the expectation notation $\mathbb{E}_\mu[f]$ for probability measures. Let $\langle \cdot,\cdot \rangle$ denote the Euclidean inner product.
 
\begin{proof}[Proof of \Cref{prop:metric}]
(i,$\Leftarrow$): Taking $T$ as in the statement, let $\mu' = T^{-1}_\sharp((T_\sharp \mu) \wedge \nu)$. This measure captures the mass of $\mu$ which is paired with $\nu$ by the isometry $T$. By the TV bound, $\mu'(\cX) \geq 1-\eps$, and so the partial coupling $\pi=(\Id,T)_\sharp \mu'\cdot \frac{1-\eps}{\mu'(\cX)}$ is feasible for the $\PGW$ problem. Thus, $\PGW(\mu,\nu) \leq \GW(\pi,\pi) = 0$. (i, $\Rightarrow$): Let $\mu'$ and $\nu'$ be the marginals of the optimal partial coupling. Since $\GW(\mu',\nu') = 0$, there is an isometry $T$ s.t.\ $\|T_\sharp \mu' - \nu'\|_\tv = 0$, and thus $\|T_\sharp \mu - \nu\|_\tv \leq \eps$.
(ii) is immediate from definition.
(iii): Let $\pi_1$ be an optimal partial coupling for $\PGW(\mu,\kappa)$, and decompose its marginals as $\mu_1 = \mu-\mu_{\eps}$ and $\kappa_1 = \kappa-\kappa_{\eps}$, where $\mu_\eps \in \cM_+(\cX), \nu_\eps \in \cM_+(\cY)$ satisfy $\mu_\eps(\cX) = \nu_\eps(\cY) = \eps$. Let $\pi_2$ be an optimal partial coupling for $\mathsf{GW}^{\delta}(\kappa,\nu)$, and similarly decompose its marginals as $\kappa_2 = \kappa-\kappa_{\delta}$ and $\nu_2 = \nu-\nu_{\delta}$. We divide $\kappa$ into four components: $A = \kappa_1-\kappa_1 \wedge \kappa_2$, $B = \kappa_1 \wedge \kappa_2$, $C = \kappa_2-\kappa_1 \wedge \kappa_2$, and $D = \kappa-\kappa_1-\kappa_2+\kappa_1 \wedge \kappa_2$. Notice $\kappa = A+B+C+D$, $(C+D)(\cY) = \kappa_{\eps}(\cY) = \eps$, and $(A+D)(\cY) = \kappa_{\delta}(\cY) = \delta$. Also, $B(\cY) \geq \kappa(\cY) - (C+D)(\cY) - (A+D)(\cY) = 1-\eps-\delta$. Next we decompose $\pi_1$ and $\pi_2$ into positive measures $\pi_1 = \pi_{1,A}+\pi_{1,B}$ and $\pi_2 = \pi_{B,2}+\pi_{C,2}$ such that $\pi_{1,A/B}$ have right marginals $A/B$ and $\pi_{B/C,2}$ have left marginals $B/C$. Gluing together $\pi_{1,B}$ and $\pi_{B,2}$ gives $\sigma \in \cM_+(\cX \times \cY \times \cZ)$ such that $(\Pi_{\cX,\cY})_\sharp\sigma = \pi_{1,B}$ and $(\Pi_{\cY,\cZ})_\sharp \sigma = \pi_{B,2}$, where $\Pi$ is the projection map. Take $\gamma = (\Pi_{\cX,\cZ})_\sharp\sigma \cdot \frac{1-\eps-\delta}{B(Y)}$. This is a feasible partial coupling for $\smash{\GW^{\eps+\delta}(\mu,\nu)}$, so we have
\begin{align*}
    \GW^{\eps+\delta}(\mu,\nu) &\leq \GW(\pi_{1,B}(\cdot,\cY),\pi_{B,2}(\cY,\cdot)) \cdot \tfrac{1-\eps-\delta}{B(\cY)} \\
    &\leq \GW(\pi_{1,B}(\cdot,\cY),\pi_{1,B}(\cX,\cdot)) + \GW(\pi_{1,B}(\cX,\cdot),\pi_{B,2}(Y,\cdot)) \\
    &\leq G(\pi_{1,B},\pi_{1,B}) + G(\pi_{B,2},\pi_{B,2})\\
    &\leq G(\pi_1,\pi_1) + G(\pi_2,\pi_2) \\
    &= \PGW(\mu,\kappa) + \GW^{\delta}(\kappa,\nu).
\end{align*}
The second inequality uses the $\GW$ triangle inequality.    
\end{proof}

\begin{definition}[Resilience w.r.t.\ $\GW$]\label{def:res}
For $0 \leq \eps \leq 1$, define the resilience of $\mu \in \cP(\cX)$ w.r.t. GW by $\rho(\mu,\eps) \defeq \sup_{\mu' \in \cP(\cX):\mu' \leq \frac{1}{1-\eps}\mu} \GW(\mu',\mu)$.

\end{definition}

\begin{lemma}
\label{lem:bound_res}
For $\eps \in [0,1/3)$, we have
\begin{align*}
    |\GW(\mu,\nu) - \PGW(\tilde{\mu}_n,\tilde{\nu}_n)| &\leq \GW(\mu,\hat{\mu}_n) + \GW(\nu,\hat{\nu}_n) + \rho(\mu,3\eps) + \rho(\nu,3\eps) + 3\eps\GW(\mu,\nu).
\end{align*}
\end{lemma}
\begin{proof}
By the approximate triangle inequality for $\GW$, we have
\begin{align*}
        \GW^{3\eps} (\mu,\nu) &\leq \GW(\mu,\hat{\mu}_n) + \PGW(\hat{\mu}_n, \tilde{\mu}_n) + \PGW(\tilde{\mu}_n, \tilde{\nu}_n) + \PGW(\hat{\nu}_n, \tilde{\nu}_n) + \GW(\nu,\hat{\nu}_n)\\
        &=\GW(\mu,\hat{\mu}_n) + \GW(\nu,\hat{\nu}_n) + \PGW(\tilde{\mu}_n, \tilde{\nu}_n). 
\end{align*}
From definition of resilience, we obtain
\begin{align*}
        \mathsf{GW}^{3\eps} (\mu,\nu) &= (1-3\eps) \inf_{\substack{\mu' \in \cP(\cX), \nu' \in \cP(\cY)\\\mu' \leq \frac{1}{1-3\eps}\mu, \,\nu' \leq \frac{1}{1-3\eps}\nu}} \mathsf{GW}(\mu',\nu')\\
        &\geq (1-3\eps) \bigl(\mathsf{GW}(\mu,\nu) - \rho(\mu,3\eps) - \rho(\nu,3\eps)\bigr).
\end{align*}
Combining these inequalities gives one direction.\smallskip

For the other direction, let $\mu' = \hat{\mu}_n \wedge \tilde{\mu}_n$ and $\nu' = \hat{\nu}_n \wedge \tilde{\nu}_n$. Notice $\mu'$ and $\nu'$ have mass greater or equal to $1-\eps$. Denoting their total masses by $m_{\mu'}$ and $m_{\nu'}$, respectively, we have
\begin{align*}
    \PGW&(\tilde{\mu}_n,\tilde{\nu}_n) \leq \GW\left(\tfrac{(1-\eps)\mu'}{m_{\mu'}}, \tfrac{(1-\eps)\nu'}{m_{\nu'}}\right)\notag\\
    \leq& \left(\!\GW\left(\tfrac{\mu'}{m_{\mu'}}, \mu\right) \!+\! \GW\left(\tfrac{\nu'}{m_{\nu'}},\nu\right) \!+\! \GW(\mu,\nu)\right) (1\!-\!\eps)\numberthis\label{otherdirection}.
\end{align*}
To bound $\smash{\GW(\mu'/m_{\mu'}, \mu)}$, let $\pi_1$ be an optimal coupling for $\GW(\hat{\mu}_n,\mu)$ and decompose it as $\pi_1  = \pi_1^{\mu'\rightarrow\mu} + \pi_1^{(\hat{\mu}_n-\mu')\rightarrow\mu}$, where the superscript denotes the support of the marginals. Write $\mu''$ for the right marginal of $\pi_1^{\mu'\rightarrow\mu}$, satisfying $\mu''\leq \mu$ and $\mu''(\mathcal{X})=m_{\mu'}\geq 1-\eps$. Thus, $\mu'' \cdot \frac{1}{m_{\mu'}} \leq \frac{1}{1-\eps} \mu$, and so
\begin{align*}
    \GW\left(\tfrac{\mu'}{m_{\mu'}}, \mu\right) &\leq \GW\left(\tfrac{\mu'}{m_{\mu'}}, \tfrac{\mu''}{m_{\mu'}}\right) + \GW\left(\tfrac{\mu''}{m_{\mu'}},\mu\right) \\
    &\leq G(\pi_1^{\mu'\rightarrow\mu},\pi_1^{\mu'\rightarrow\mu}) \cdot \frac{1}{m_{\mu'}}+\rho(\mu,\eps)\\
    &\leq \GW(\mu,\hat{\mu}_n)\cdot \frac{1}{1-\eps}+\rho(\mu,3\eps).
\end{align*}
Similarly, we have
\begin{align*}
     \GW\left(\tfrac{\nu'}{m_{\nu'}}, \nu\right) \leq \GW(\nu,\hat{\nu}_n) \cdot \frac{1}{1-\eps}+\rho(\nu,3\eps).
\end{align*}
Plugging back into \eqref{otherdirection} we have
\begin{align*}
    \PGW(\tilde{\mu}_n,\tilde{\nu}_n) - \GW(\mu,\nu)
    \leq & \GW(\mu,\hat{\mu}_n) + \GW(\nu,\hat{\nu}_n) + \rho(\mu,3\eps) + \rho(\nu,3\eps),
\end{align*}
which gives the other direction.
\end{proof}

\begin{lemma}
\label{lem:resilienceK}
For $\mu \!\in\! \cG_k(\sigma)$, $k \!\geq\! 4$, and $0 \!\leq\!\eps \!\leq\! 0.99$, we have $\rho(\mu,\eps) \!\lesssim\! \sigma^2\eps^{\frac{1}{2}-\frac{2}{k}}$.   
\end{lemma}
\begin{proof}
For $\lambda \in (0,1)$ and $\beta \in \cP(\cX)$ with $\beta \leq \frac{1}{\lambda}\mu$, we have
\begin{equation}
\label{eq:holder-bd}
    \mathbb{E}_\beta[\|X\|^4] \leq \lambda^{-1} \sup_{E:\mu(E) = \lambda} \mathbb{E}_\mu[\|X\|^4 \mathds{1}_E] \leq \sigma^4 \lambda^{-4/k}.
\end{equation}
Now, for $\nu \in \cP(\cX)$ with $\nu \leq \frac{1}{1-\eps}\mu$, write $\mu = (1-\eps)\nu + \eps \alpha$ for $\alpha \in \cP(\cX)$. Let $\pi_{\alpha,\nu}$ be an optimal coupling for $\GW(\alpha,\nu)$. Using $\pi = \eps \pi_{\alpha,\nu} + (1-\eps) (\Id, \Id)_\sharp \nu \in \Pi(\mu,\nu)$, we bound
\begin{equation*}
    \GW(\mu,\nu)^2 \leq \eps^2 \GW(\alpha,\nu)^2 + 2G\bigl(\eps \pi_{\alpha,\nu},(1-\eps) (\Id, \Id)_{\sharp} \nu\bigr)^2.
\end{equation*}
We deal with the two terms separately. First, we bound
\begin{align*}
    \GW(\alpha,\nu) &\leq \GW(\alpha,\delta_0) +  \GW(\nu,\delta_0) \\
    &= \mathbb{E}_{\alpha\otimes\alpha}[\|X_1-X_2\|^{4}]^{\frac{1}{2}} + \mathbb{E}_{\nu\otimes\nu}[\|X_1-X_2\|^{4}]^{\frac{1}{2}}\\
    &\leq 2 \sup_{\beta \leq \frac{1}{1-\tau}\mu} \mathbb{E}_{\beta\otimes\beta}[\|X_1-X_2\|^{4}]^{\frac{1}{2}}
\end{align*}
where $\tau = \eps \lor (1-\eps)$. By $(x-y)^4 \leq 8(x^4+y^4)$ and \eqref{eq:holder-bd},
\begin{align*}
        \mathbb{E}_{\beta\otimes\beta}[\|Y_1-Y_2\|^{4}]^{\frac{1}{2}} \lesssim \mathbb{E}_{\beta}[\|Y\|^{4}]^{\frac{1}{2}} \lesssim \sigma^2 (1-\tau)^{-\frac{2}{k}}.
\end{align*}
Again using \eqref{eq:holder-bd}, we similarly bound the second term
\begin{align*}
    G(\eps \pi_{\alpha,\nu},(1-\eps) (\Id, \Id)_\sharp \nu)^2 &\leq \eps(1-\eps) (\alpha \otimes \nu + \nu \otimes \nu)(\|x_1 - x_2\|^4)\\
    &\lesssim \eps(\alpha + \nu)(\|x\|^4)\\
    &\lesssim \eps \sup_{\beta \leq \frac{1}{1-\tau}\mu} \mathbb{E}_\beta[\|X\|^4] \\
    &\lesssim \sigma^{4} \eps (1-\tau)^{-\frac{4}{k}}.
\end{align*}
Combining, we obtain
\begin{align*}
    \GW(\mu,\nu) &\lesssim \sqrt{\sigma^4\eps^{2}(1-\tau)^{-\frac{4}{k}} + \sigma^{4} \eps (1-\tau)^{-\frac{4}{k}}} \lesssim \sigma^2\eps^{\frac{1}{2}-\frac{2}{k}}.
\end{align*}
using that $\tau^{-4/k} \lesssim \eps^{-4/k}$ for $\eps \leq 0.99$.
\end{proof}

\begin{proof}[Proof of \Cref{thm:sym_pop-limit}]
The upper bound follows directly from \Cref{lem:bound_res} (as $n \to \infty$), \Cref{lem:resilienceK}, and the fact that
\begin{align*}
    \GW(\mu,\nu) 
    &\leq \left(\mathbb{E}_{\mu\otimes\mu}[\|X_1-X_2\|^{4}] + \mathbb{E}_{\nu\otimes\nu}[\|Y_1-Y_2\|^{4}]\right)^{\frac{1}{2}}  \\
    &\leq \left( 16\mathbb{E}_{\mu}[\|X\|^{4}] + 16\mathbb{E}_{\nu}[\|Y\|^{4}]\right)^{\frac{1}{2}}
    \leq 8 \sigma^2,
\end{align*}
so every term is no greater than $O(\sigma^2\eps^{\frac{1}{2}-\frac{2}{k}})$.\smallskip

For the lower bound, consider $\mu_1 = (1-\eps) \delta_0 + \eps \delta_a$ where $a = \sigma/\eps^{\frac{1}{k}}$ so that $\mu_1 \in \cG_k(\sigma)$. We then have
\begin{align*}
    R_{\infty}(\cG_k(\sigma),\eps) &\geq \inf_{\mathsf{T}} \sup_{\substack{\mu \in \{\mu_1,\delta_0\}, \nu =\delta_0,\\\tilde{\mu} = \mu_1, \tilde{\nu} = \delta_0}}
    |\GW(\mu,\nu)-\mathsf{T}(\tilde{\mu},\tilde{\nu})|\\
    &\geq \frac{1}{2}\GW(\mu_1,\delta_0) = \frac{1}{2}(2(1-\eps))^{\frac{1}{2}} \sigma^2 \eps^{\frac{1}{2}-\frac{2}{k}}.
\end{align*}
The first inequality follows since $\mu_1,\delta_0 \in \cG_k(\sigma)$ and $\|\mu_1 - \delta_0\|_\tv \leq \eps$. The second uses that the best estimate is to take the midpoint of the range of all possible distances (we have only two  distances: $0$ and $\GW(\mu_1,\delta_0)$).
\end{proof}

\begin{proof}[Proof of \Cref{cor:pop-limit-examples}]
Upper bounds follow directly from inclusions to bounded moment families. Indeed, Lemma 6 in \cite{nietert2022statistical} gives that $\overline{\cG}_k(\sigma)\subseteq \cG_{k}(O(\sigma\sqrt{1+d/k}))$ for all $k$, and $\cG_{\mathrm{subG}}(\sigma) \subseteq \cap_{k \geq 1} \overline{\cG}(O(\sigma \sqrt{k})) \subseteq \overline{\cG}_{\lceil \log(1/\eps) \rceil}(O(\sigma \sqrt{\log(1/\eps)}\,))$.\smallskip

The sub-Gaussian lower bound can be obtained using the same construction as in Theorem 2 in \cite{nietert2023robust}.
For $\overline{\cG}_k(\sigma)$, consider $X\sim\mu = (1-\epsilon)\delta_0+\epsilon\cN(0,c^2 I_d)$, where $c = \epsilon^{-\frac{1}{k}} \sigma k^{-\frac{1}{2}}$ and $\cN$ the Gaussian measure. Then, for each unit vector $v$, we have $\langle X, v\rangle \sim (1-\epsilon)\delta_0+\epsilon\cN(0,c^2)$, and $\mathbb{E}_\mu[|\langle X, v\rangle|^k] = \epsilon \sigma_d^k \cdot (k-1)!!\leq \sigma^k$. Further, $\GW(\mu,\delta_0) \geq (2\epsilon(1-\epsilon)\cdot \E_{\cN(0,c^2)}[\|X\|^{4}])^{\frac{1}{2}}$, where $\E_{\cN(0,c^2)}[\|X\|^{4}] = c^{4}\cdot 4\cdot \frac{\Gamma(2+d/2)}{\Gamma(d/2)} \geq c^{4}\cdot d^{2}$. Plugging back we get $\GW(\mu,\delta_0) \geq \sigma^2 \epsilon^{\frac{1}{2}-\frac{2}{k}} \cdot d/k$. Using the two-point method above, we have $R_{\infty}(\overline{\cG}_k(\sigma),\eps) \gtrsim \sigma^2 \epsilon^{\frac{1}{2}-\frac{2}{k}} \cdot d/k$. The construction in proof of \Cref{thm:sym_pop-limit} also gives $R_{\infty}(\overline{\cG}_k(\sigma),\eps) \geq \sigma^2 \epsilon^{\frac{1}{2}-\frac{2}{k}}$ since $\mu_1\in \overline{\cG}_k(\sigma)$. Combining gives the desired bound.
\end{proof}

\begin{proof}[Proof of \Cref{thm:sym_n_sample}]
Using the same approach as in the proof of Theorem $3$ in \cite{nietert2023robust}, we obtain $R_n(\cG,\eps) \gtrsim R_\infty(\cG,\eps/4)$. Further, $R_n(\cG,\eps) \geq R_n(\cG,0)$ from definition. Together this gives the lower bound.
The upper bound for specific families follows similarly to the population limit result using \Cref{lem:bound_res,lem:resilienceK} (using inclusions into $\overline{\cG}_k(\sigma)$ for the sub-Gaussian and sliced moment classes).
\end{proof}

\begin{proof}[Proof of \cref{prop:empirical-convergence}]
Fix $\mu \in \cG_k(\sigma)$, with $d > k > 2pq$. By Lemma 1 in \cite{zhang2024gromov} and H\"older's inequality, we have
\begin{align*}
    \E[\GW(\mu,\hat{\mu}_n)] &\lesssim \E[((\mu+\hat{\mu}_n)(\|x\|^4))^{\frac{1}{4}}\mathsf{W}_{4}(\mu,\hat{\mu}_n)] \\
    &\leq \E[(\mu+\hat{\mu}_n)(\|x\|^4)]^{\frac{1}{4}} \cdot \E[\mathsf{W}_{4}(\mu,\hat{\mu}_n)^{\frac{4}{3}}]^{\frac{3}{4}}\\
    &\leq (2\sigma^4)^{\frac{1}{4}}\cdot \E[\mathsf{W}_{4}(\mu,\hat{\mu}_n)^{4}]^{\frac{1}{4}}\\
    &\lesssim \sigma \cdot \E[\mathsf{W}_{4}(\mu,\hat{\mu}_n)^{4}]^{\frac{1}{4}}.
\end{align*}
Combining with standard empirical convergence results for $\mathsf{W}_4$ (e.g., \citealp{fournier2015rate}), we obtain the stated upper bound.
\end{proof}

\begin{proof}[Proof of \Cref{thm:asym_pop-limit+n_sample}]
The first lower bound can be obtained similarly to that of \Cref{thm:sym_n_sample}. Further, risk definitions imply
\begin{equation*}
    R_\infty(\cG_1,\cG_2,\eps) \!\lesssim\! R_\infty(\PGW, \cG_1,\cG_2,\eps) \!\lesssim\! R_\infty(\PGW, \cG_1\cup \cG_2,\eps).
\end{equation*}
The upper bound proof for \cref{thm:sym_pop-limit} and the lower bounds from \Cref{cor:pop-limit-examples} imply that $R_\infty(\PGW, \cG_1\cup \cG_2,\eps) \lesssim R_\infty(\cG_1,\eps) \lor R_\infty(\cG_2,\eps)$. By definitions, we can lower bound $R_\infty(\cG_1\cup \cG_2,\eps)$ by
\begin{equation*}
    R_\infty(\cG_1\cup \cG_2,\eps) \geq R_\infty(\cG_1,\eps) \lor R_\infty(\cG_2,\eps),
\end{equation*}
so we obtain $R_\infty(\cG_1\cup \cG_2,\eps) \asymp R_\infty(\PGW, \cG_1\cup \cG_2,\eps)$. Thus, to prove the first statement, it suffices to show $R_\infty(\cG_1,\cG_2,\eps) \gtrsim R_\infty(\cG_1\cup \cG_2,\eps)$.
To see this, notice for any $\mu_1\in\cG_1,\nu_1\in\cG_2$, we have
$$
\begin{aligned}
    R_{\infty}(\cG_1, \cG_2, \eps) &\geq \inf_{\sfT} \sup_{\substack{\mu \in \{\mu_1,\delta_0\}, \nu \in \{\nu_1,\delta_0\},\\\tilde{\mu} = \mu_1, \tilde{\nu} = \nu_1}}
|\mathsf{GW}(\mu,\nu)-\mathsf{T}(\tilde{\mu},\tilde{\nu})|\\
&= \frac{1}{2}\left|\max_{\substack{\mu \in \{\mu_1,\delta_0\}\\\nu \in \{\nu_1,\delta_0\}}} \GW(\mu,\nu)  - \min_{\substack{\mu \in \{\mu_1,\delta_0\}\\ \nu \in \{\nu_1,\delta_0\}}} \GW(\mu,\nu)\right|\\
&\geq \frac{1}{2}\max \{\GW(\mu_1,\delta_0),\GW(\nu_1,\delta_0)\}.
\end{aligned}
$$
Taking $\mu_1$, $\nu_1$ be those that achieved lower bound in the symmetric case, we obtain
$R_{\infty}(\cG_1, \cG_2, \eps) \gtrsim R_{\infty}(\cG_1,\eps) \lor R_{\infty}(\cG_2,\eps) = R_\infty(\cG_1 \cup \cG_2,\eps)$.\smallskip

The second statement follows directly by \Cref{lem:bound_res}.
\end{proof}

\section{Concluding Remarks}\label{sec:summary}

We studied the problem of robust estimation of the quadratic GW distance under the $\eps$-TV contamination model. The proposed estimator computed the partial GW distance between the contaminated observations. For chosen distribution families, the estimator was shown to be minimax optimal in the population limit and near-optimal when sampling is present. Notably, the robust estimation risk scales with the complexity of the union of the two distribution families, as opposed to the empirical estimation setting from clean data, where convergence rates adapt to the lower-dimensional dataset \cite{zhang2024gromov}. The results pose the partial GW problem as a natural outlier-robust surrogate of the classical setting, providing an operational meaning through the optimality for the robust estimation setting.

\clearpage
\bibliographystyle{unsrtnat}
\bibliography{ref}

\end{document}